\documentclass{svjour3}                     
\pdfoutput=1
\smartqed  
\usepackage{graphicx}

\usepackage{mathptmx}      
\usepackage{epsfig}
\usepackage{amsmath}
\usepackage{amssymb}
\usepackage{color}
\newcommand{\ds}{\displaystyle}
\newcommand{\sat}{\mbox{sat}}
\newcommand{\std}{\operatorname{std}}

\newcommand{\js}[1]{{#1}}
\newcommand{\ar}[1]{{#1}}
\renewcommand{\d}{\mathop{}\!\mathrm{d}}
\journalname{Bioprocess and Biosystems Enginnering}
\begin{document}

\title{A method for the reconstruction of unknown non-monotonic growth functions in the chemostat
\thanks{JS' research is supported by the EPSRC grant EP/J010820/1.}
}


\author{Jan Sieber \and Alain Rapaport \and Serafim Rodrigues \and Mathieu Desroches}


\institute{J. Sieber \at
               Univ. of Exeter, U.K.
              \email{J.Sieber@exeter.ac.uk}           
           \and
           A. Rapaport \at
              UMR INRA/SupAgro MISTEA and EPI INRA/INRIA MODEMIC,
              Montpellier, France\\
              \email{rapaport@supagro.inra.fr}           
           \and
           S. Rodrigues \at
              CN-CR, Univ. of Plymouth, U.K.\\
              \email{serafim.rodrigues@plymouth.ac.uk}           
           \and
           M. Desroches \at
              EPI SISYPHE, INRIA Rocquencourt, France\\
              \email{Mathieu.Desroches@inria.fr}           
}

\date{Received: date / Accepted: date}

\titlerunning{Reconstruction of unknown non-monotonic growth functions}
\maketitle

\begin{abstract}
  We propose an adaptive control law that allows one to identify
  unstable steady states of the open-loop system in the single-species
  chemostat model without the knowledge of the growth function. We
  then show how one can use this control law to trace out
  (reconstruct) the whole graph of the growth function. The process of
  tracing out the graph can be performed either continuously or
  step-wise. We present and compare both approaches.  Even in the case
  of two species in competition, which is not directly accessible with
  our approach due to lack of controllability, feedback control
  improves identifiability of the non-dominant growth rate.
\end{abstract}

\section{Introduction}

We recall the classical chemostat model \cite{SW95} for a single
species (biomass $b$) consuming a substrate (mass $s$):

\begin{equation}
\label{chemostat}
\left\{\begin{array}{lll}
\dot s & = & -\mu(s)b+D(s_\mathrm{in}-s)\\
\dot b & = & \mu(s)b-Db
\end{array}\right.
\end{equation}
where the \emph{dilution rate} $D$ (the \emph{input}) is the
manipulated variable, which takes values in a bounded positive
interval $[D_{\min},D_{\max}]$, and $\mu(\cdot)$ is a non-negative
Lipschitz continuous function with $\mu(0)=0$
.

We consider here the following scenario: the function $\mu(\cdot)$ is
unknown and possibly non-monotonic. Our objective is to reconstruct
the graph of the function $\mu(\cdot)$ on the domain
$(0,s_\mathrm{in})$ by varying the input $D$ in time. On-line
measurements are only available for the variable $s$ (that is, $s$ is
the \emph{output}). This setup is realistic for experimental
investigations such as in \cite{AHSA03}, however, demonstrations in
this paper are based entirely on simulations of models such as
system~\eqref{chemostat}.
The present paper analyzes and expands the ideas initially
proposed by the authors in the conference paper \cite{SRRD12}.

\noindent {\em Remark:} Using model~\eqref{chemostat} tacitly assumes
that the yield coefficient of the bio-conversion is known. This is why
$\mu(s)b$ appears with the same pre-factor $1$ (once positive, and
once negative) in both equations of \eqref{chemostat} without loss of
generality.

The problem of kinetics estimation in biological and biochemical
models has been widely addressed in the literature
(\cite{AW78,H82,HR82,R82,DB84,DB86,DP88,BD90,PM90,LF91,BSSR94,VK96,DP97,VCV97,KE98,VCV98,PFFD00,D03}), either as a parameter
estimation problem (one chooses a priori an analytical expression of
the function $\mu(\cdot)$), or as an on-line estimation of the
kinetics (one aims at determining $\mu(s(t))$ at the current time
$t$).  The theoretical identifiability of the graph of $\mu(\cdot)$
has been thoroughly studied in \cite{BG03}. In this paper, a practical
method has been proposed to reconstruct the graph of $\mu(\cdot)$,
based on a Kalman observer under the approximation that the function
$z(t)=\mu(s(t))b(t)$ has a third time derivative equal to zero.  

Here, we propose a different method that does not make any
approximation of the dynamics. Our method exploits that it is
sufficient to find the complete branch of equilibria of system
\eqref{chemostat} to identify the graph $\mu(\cdot)$. This reduces the
system identification problem to a combination of two problems:
finding the equilibria (a root-finding problem) and stabilizing them
(a feedback control problem). Both of the latter two problems are in
theory easily solvable with standard methods as we will explain and
illustrate in
Sections~\ref{section-basic-feedback}--\ref{sec:hybrid}. The most
difficult obstacle in practice is the implementation of a real-time
feedback loop measuring $s$ and adapting $D$ with sufficient accuracy.

When the growth function is monotonic, a common way to reconstruct
points on the graph of the growth function $\mu(\cdot)$ is to design a
series of experiments fixing the dilution rate $D$ with different
values and wait until the system settles to a steady state
$(s^{\star},b^{\star})$ \cite{BD90}. As long as $D$ is less than
$\mu(s_\mathrm{in})$, it is well known that the dynamics converges to a
unique positive equilibrium that satisfies $\mu(s^{\star})=D$ (see for
instance \cite{SW95}). This technique requires the steady state to be stable
in open loop, and consequently cannot reconstruct any part of the
graph of a function $\mu(\cdot)$ where $\mu$ is non-increasing (such
as the example shown schematically in Figure~\ref{fig:stability}). 
Furthermore, the global convergence of this method is not satisfied 
in case of bi-stability, which is present in \eqref{chemostat}, 
with non-monotonic growth functions $\mu$ (see again \cite{SW95}).
\begin{figure}[h]
  \includegraphics[width=0.4\columnwidth]{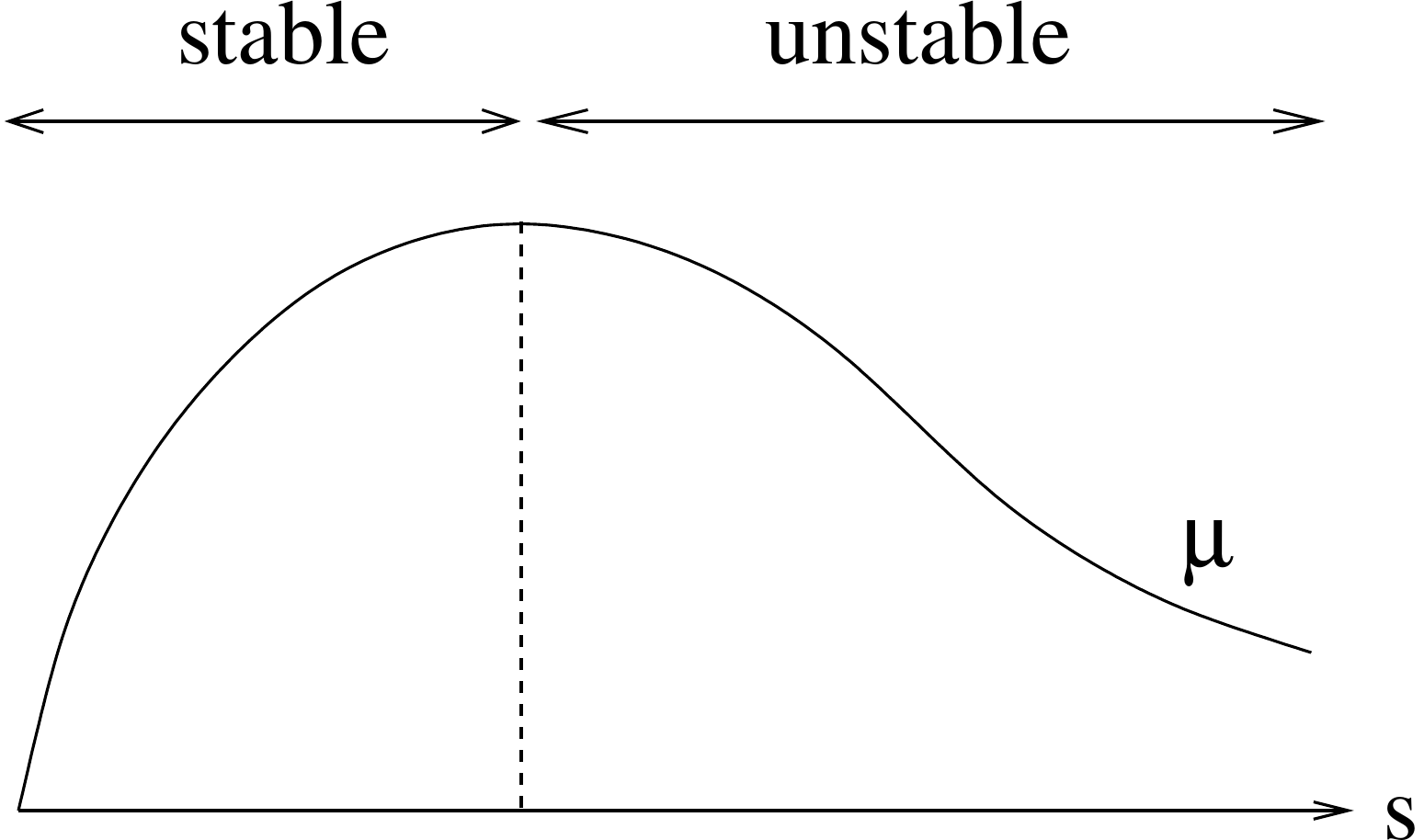}
  \caption{Domains of stability and instability in open-loop}
  \label{fig:stability}
\end{figure}

An alternative approach is to fix a value of
$s$, say $\bar s$, and design an adaptive control law $D(\cdot)$ that
stabilizes the system about the steady state $(\bar s,s_\mathrm{in}-\bar s)$,
with the value of $D$ converging to $\mu(\bar s)$.
Several adaptive control laws have been proposed in the literature for
this problem. Nonlinear feedbacks require the knowledge of the growth
function $\mu(\cdot)$, such as linearizing controls \cite{BD90,KH97}
Extensions that are robust with respect to uncertainty on $\mu(\cdot)$
have been proposed \cite{RH02} but do not provide the precise
reconstruction of $\mu(\bar s)$.  Several nonlinear PI based
controllers, that do not require the precise knowledge of
$\mu(\cdot)$, have been also proposed \cite{RC94,SC99}, but saturation
and windup is often an issue (see \cite{JBV99,KD99} in similar
frameworks).  In \cite{AHSA03}, a dynamical output feedback has been
proposed to globally stabilize such dynamics without the knowledge of
$\mu(\cdot)$ and under the constraint $D\in[D_{\min},D_{\max}]$, but
it requires the growth function to be monotonic. More recently, a
saturated PI controller coupled with an observer, dedicated to the
non-monotonic case, have been proposed to stabilize the dynamics about
a nominal point that maximizes the biomass production \cite{SAL12}.

\ar{As the exploration of the unstable part of an unknown non-monotonic
growth function requires an adaptive feedback, we propose in the
present work to take advantage of an adaptation scheme for exploring
(at least) a part of the graph instead of a limited number of set-points.}
We first consider that it can be useful to introduce a
feedback control loop into \eqref{chemostat} to identify the growth
function $\mu$ of the \emph{open-loop} system \eqref{chemostat} (that
is, \eqref{chemostat} with constant input $D$). The feedback control law is
initially a simple saturated proportionate controller:
\begin{equation}
  \label{simple-feedback}
  D(s,\bar D,\bar s)=\sat_{[D_{\min},D_{\max}]}
  \left(\bar D - G_{1}(s-\bar s)\right)\mbox{,}
\end{equation}
where $\bar D$ and $\bar s$ are reference values, and $G_1>0$ is the
linear control gain.  To ensure realistic values for the input $D$,
feedback law \eqref{simple-feedback} encloses the linear feedback rule
into the saturation function
\begin{displaymath}
  \sat_{[D_{\min},D_{\max}]}(x)=
  \begin{cases}
    D_{\max} &\mbox{if $x> D_{\max}$,}\\
    x &\mbox{if $x\in[D_{\min}, D_{\max}]$,}\\
    D_{\min} &\mbox{if $x< D_{\min}$,}
  \end{cases}
\end{displaymath}
where the limits $D_{\min}$ and $D_{\max}$ are the extreme dilution
rates that can be achieved experimentally. In sections
\ref{section-adaptive} and \ref{sec:hybrid} we will then explore
adaptation rules for the reference values $(\bar s,\bar D)$ which
ensure that asymptotically for $t\to\infty$ the input satisfies
\begin{equation}
  \label{non-invasive1}
  D(s,\bar D,\bar s)=\bar D \,
\end{equation}
or, equivalently, the output satisfies
\begin{equation}
\label{non-invasive2}
  s=\bar s\mbox{.}
\end{equation}
\js{If \eqref{non-invasive2} is satisfied in the limit $t\to\infty$
  then the controlled input $D$ in \eqref{simple-feedback} equals the
  open-loop value $\bar D$ again (the feedback term $G_1(s-\bar s)$ in
  \eqref{simple-feedback} vanishes). We call feedback control that
  vanishes asymptotically \emph{non-invasive}.} The result is a new
adaptive control law that stabilizes the dynamics about any desired
equilibrium point without requiring a priori knowledge of its
location, and whatever is the monotonicity of the growth function. One
requirement on the adaptive law is that it should work uniformly well
around a local maximum of $\mu$ (non-invasive feedback laws such as
\js{(wash-out)} filtered feedback \cite{AWC94} or time-delayed
feedback \cite{P92} do not achieve this).

We note that our adaptation rules will be much simpler than classical
adaptive control laws \cite{BD90,AW94}.  Usually, adaptive control
aims to achieve a desired output regardless of changes in the
underlying system. We only adapt the reference values to make the
control input vanish and find branches of equilibria and bifurcations
of the underlying system, similar to numerical continuation
\cite{AG03}. While classical adaptive control requires system
identification (an inverse problem) at some stage, our adaptation
solves a root-finding problem, which is simpler.

Throughout our paper we assume that the output $s$ can be sampled
and the input $D$ can be adjusted in quasi real-time. If the sampling
period $T$ is not negligible, the approach presented here can still be
applied. However, one then faces the problem that feedback
stabilization of an unstable equilibrium at output $s$ becomes
sensitive with amplification factor $\sim\exp(-\mu'(s)T)$ to
disturbances (note that $\mu'(s)<0$ for unstable equilibria).

We show in Sections~\ref{section_continuation} and \ref{sec:hybrid}
that the feedback law \eqref{simple-feedback} can be combined with an
adaption rule for $(\bar s,\bar D)$ to reconstruct the graph of the
growth function, even in the case of non-monotonic growth
functions. Section~\ref{section_continuation} presents a dynamic
adaption, whereas Section~\ref{sec:hybrid} introduces a step-wise
adaptation. In Section~\ref{sec:twospecies} we investigate the case of
two species that compete for the same common substrate.

\section{Global stability of the simple feedback law
  \eqref{simple-feedback}}
\label{section-basic-feedback}

\begin{figure}[t]
  \includegraphics[width=0.6\columnwidth]{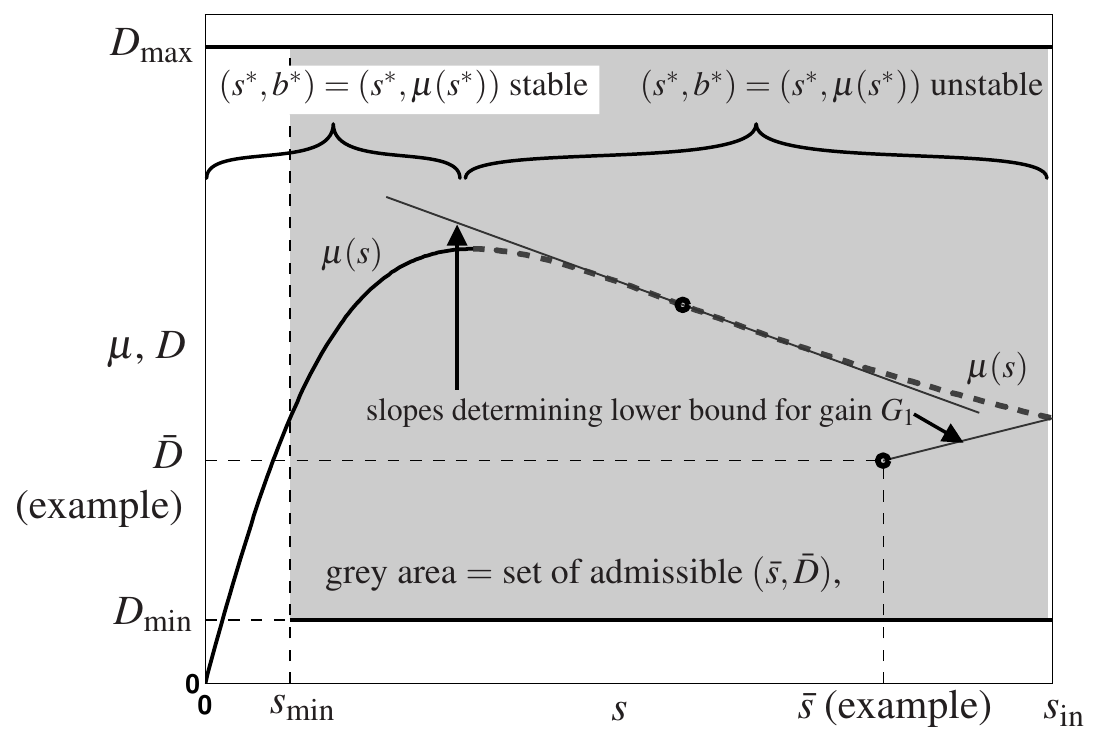}
  \caption{Sketch illustrating shape of growth function $\mu$ and set of admissible reference values}
  \label{fig:sketch}
\end{figure}
Let us first prove that the feedback law \eqref{simple-feedback} is,
within reasonable limits, globally stabilizing. Suppose that we choose
the reference value $\bar s$ from an interval
$[s_{\min},s_\mathrm{in})\subset(0,s_\mathrm{in})$, and that the
limits on the input cover the growth function $\mu$ on this interval:
\begin{align}
  \label{eq:muboundlow}
  D_{\min}&< \mu(s)\mbox{\quad for all $s\in[s_{\min},s_\mathrm{in}]$,}\\
  \label{eq:muboundup}
  D_{\max}&>\mu(s)\mbox{\quad for all $s\in[0,s_\mathrm{in}]$.}
\end{align}
These conditions mean that the graph of $\mu$ does not cross the thick
parts of the horizontal lines $D_{\min}$ and $D_{\max}$ bounding the
grey area in Figure~\ref{fig:sketch} from below and above.

\begin{proposition}
\label{proposition1}
Suppose that the reference values $(\bar s,\bar D)$ in feedback law
\eqref{simple-feedback},
\begin{displaymath}
  D(t)=\sat_{[D_{\min},D_{\max}]} \left(\bar D
  - G_{1}(s(t)-\bar s)\right)\mbox{,}
\end{displaymath}
are chosen from the rectangle
$[s_{\min},s_\mathrm{in})\times[D_{\min},D_{\max}]$, that the growth
function $\mu$ satisfies \eqref{eq:muboundlow}--\eqref{eq:muboundup},
and that the gain $G_1$ is chosen sufficiently large, that is,
\begin{align}
  \label{eq:g1boundmuprime}
  G_1&>-\min_{s\in[0,s_\mathrm{in}]}\mu'(s)\mbox{, and}\\
  \label{eq:g1boundsin}
  G_1&>-\left[\frac{\mu(s_\mathrm{in})-
      \bar D}{s_\mathrm{in}-\bar s}\right]\mbox{.}
\end{align}
Then the controlled system~\eqref{chemostat} with $D=D(s,\bar D,\bar
s)$ has a stable equilibrium
$(s_\mathrm{eq},b_\mathrm{eq})\in[0,s_\mathrm{in})\times(0,\infty)$,
which attracts all initial conditions
$(s(0),b(0))\in[0,s_\mathrm{in})\times(0,\infty)$.
\end{proposition}

\begin{proof} If $D>0$, and the growth function $\mu$
satisfies $\mu(0)=0$ and, for $s>0$, $\mu(s)>0$ then the set
\begin{displaymath}
R=\{(s,b): s\in[0,s_\mathrm{in}), b>0\}
\end{displaymath}
is positively invariant (that is, trajectories starting in $R$ will
stay in $R$ for all positive times). Furthermore, all trajectories
starting in $R$ approach the subspace (called stochiometric set in
\cite{SAL12})
\begin{displaymath}
T=\{(s,b)\in R: s+b=s_\mathrm{in}\}
\end{displaymath}
with rate at least $D_{\min}$ forward in
time. This implies that it is sufficient to check if all trajectories
in $T$ converge to a unique equilibrium. On $T$ the equation of motion
can be expressed as a differential equation for $s$ only:
\begin{equation}
  \label{eq:invsubspace}
  \dot s=[D(s,\bar D,\bar s)-\mu(s)][s_\mathrm{in}-s]=
  \left[\sat_{[D_{\min},D_{\max}]}
    \left(\bar D - G_{1}(s-\bar s)\right)-\mu(s)\right][s_\mathrm{in}-s]
  \mbox{.}
\end{equation}
First, let us check that the equilibrium at $s=s_\mathrm{in}$ is
unstable. The term $-G_1(s_\mathrm{in}-\bar s)$ is negative such that
$\bar D-G_1(s_\mathrm{in}-\bar s)<D_{\max}$ for all admissible $\bar
D$.  Assumption \eqref{eq:g1boundsin} guarantees that $\bar
D-G_1(s_\mathrm{in}-\bar s)<\mu(s_\mathrm{in})$. Assumption
\eqref{eq:muboundlow} guarantees that also
$D_{\min}<\mu(s_\mathrm{in})$. Hence,
\begin{displaymath}
D(s_\mathrm{in},\bar D,\bar s)-\mu(s_\mathrm{in})<0
\end{displaymath}
for all admissible $(\bar s, \bar D)$. Thus, the prefactor of
$s_\mathrm{in}-s$ in \eqref{eq:invsubspace} is negative such that the
equilibrium at $s_\mathrm{in}$ is unstable for all admissible $(\bar
s, \bar D)$.

Since $\dot s>0$ at $s=0$, there must be other equilibria of
\eqref{eq:invsubspace} in $(0,s_\mathrm{in})$, which are given as
solutions $s_\mathrm{eq}$ of $D(s_\mathrm{eq},\bar D, \bar
s)=\mu(s_\mathrm{eq})$.  Now let us check indirectly that none of the
equilibria can satisfy $D_{\min}=\mu(s_\mathrm{eq})$. 

Assume that \eqref{eq:invsubspace} had an equilibrium $s_\mathrm{eq}$
with $D_{\min}=\mu(s_\mathrm{eq})$. Then $s_\mathrm{eq}$ has to be
less than $s_{\min}$ due to assumption \eqref{eq:muboundlow}. However,
if $s_\mathrm{eq}<s_{\min}$, then $\bar D-G_1(s_\mathrm{eq}-\bar
s)>\bar D\geq D_{\min}$ for all admissible $(\bar s,\bar D)$. Hence
$D(s_\mathrm{eq},\bar D,\bar s)>D_{\min}$ (recall that
$D_{\min}=\mu(s_\mathrm{eq})$ by assumption of the indirect proof)
such that $D(s_\mathrm{eq},\bar D,\bar s)-\mu(s_\mathrm{eq})>0$, which
means that $s_\mathrm{eq}$ cannot be equilibrium, establishing the
contradiction.

Assumption \eqref{eq:muboundup} excludes that equilibria with
$\mu(s_\mathrm{eq})=D_{\max}$ exist, hence all remaining equilibria
$s_\mathrm{eq}\in(0,s_\mathrm{in})$ must satisfy
\begin{equation}\label{eq:seqdef:imp}
  \bar D-G_1(s_\mathrm{eq}-\bar s)=\mu(s_\mathrm{eq})\mbox{.}
\end{equation}
Condition \eqref{eq:g1boundmuprime} ensures that this equation has a
unique solution and that this solution corresponds to a stable
equilibrium (which must be in $(0,s_\mathrm{in})$ because the boundaries of 
$(0,s_\mathrm{in})$ are inflowing for \eqref{eq:invsubspace}).\qed
\end{proof}

Proposition \ref{proposition1} ensures that the output $s_\mathrm{eq}$ of the
controlled system~\eqref{chemostat} with \eqref{simple-feedback},
after transients have decayed, is a well-defined smooth function of
the parameters $(\bar s,\bar D)$ as long as $(\bar s,\bar D)$ are
chosen from $(s_{\min},s_{\max})\times(D_{\min},D_{\max})$. We express
this fact by using the bracket notation:
\begin{equation}\label{eq:seqdef}
s_\mathrm{eq}(\bar s,\bar D)=\lim_{t\to\infty}s(t)
\mbox{\quad where $s$ is output of \eqref{chemostat},\,\eqref{simple-feedback}.}
\end{equation}
The function $s_\mathrm{eq}$ can be evaluated at any admissible point
by setting the parameters $(\bar s,\bar D)$ in the definition
\eqref{simple-feedback} of the feedback rule, waiting until the
transients of \eqref{chemostat} have settled, and then reading off the
output $s$. Equilibria of the uncontrolled system can then, according
to \eqref{non-invasive2}, be found as roots of $s_\mathrm{eq}(\bar
s,\bar D)-\bar s$. More specifically, we know that, for any admissible
$\bar s$,
\begin{equation}
  \bar D=\mu(\bar s)\mbox{\quad if and only if $s_\mathrm{eq}(\bar s,\bar D)=\bar s$.}
  \label{eq:muident}
\end{equation}
Relation~\eqref{eq:muident} permits us to identify $\mu(\bar s)$ as
the unique root of $s_\mathrm{eq}(\bar s,\cdot)-\bar s$. Sections
\ref{section-adaptive}-\ref{sec:hybrid} will explore two
strategies to find this root for a range of admissible $\bar s$
efficiently.

\section{An adaptive control scheme}
\label{section-adaptive}

The first strategy is a dynamic feedback that comes on top of the
feedback law \eqref{simple-feedback} for $D$. We treat $\bar D$ not as
a parameter but introduce an additional dynamical equation for $\bar
D$, achieving local convergence of the output $s$ to any reference
value $\bar s\in (0,s_\mathrm{in})$ without the knowledge of the
growth function $\mu$. Then the asymptotic value of $\bar D$ allows
one to reconstruct the value $\mu(\bar s)$.

\begin{proposition}
\label{proposition2}
Fix a number $\bar s \in (0,s_\mathrm{in})$ and take numbers $D_{\min}$, $D_{\max}$ that fulfill $0<D_{\min} < \mu(\bar s) < D_{\max}$.
Then the dynamical feedback law
\begin{equation}
\label{dyn-feedback}
\begin{split}
D(s,\bar D)&=\sat_{[D_{\min},D_{\max}]}\left(\bar D - G_{1}(s-\bar s)\right)\\
\frac{\d}{\d t}\bar D &= -G_{2}(s-\bar s)(\bar D-D_{\min})(D_{\max}-\bar D)
\end{split}
\end{equation}
exponentially stabilizes the system \eqref{chemostat} locally about
$(s,b)=(\bar s,s_\mathrm{in}-\bar s)$, for any positive constants
$(G_{1},G_{2})$ such that $G_{1}> -\mu^{\prime}(\bar s)$.  Furthermore
one has 
\[
\lim_{t\to+\infty}\bar D(t)=\mu(\bar s)
\]
\end{proposition}

{\em Remark.} Our adaptive control is in this case similar to a
classical PI controller. The quantity $\bar D$ is playing the role of
the I part, but staying bounded by construction.
\ar{It is also similar to gain-scheduling methods but here
  the parameter $\bar D$ is evolving continuously, as a state variable.}

\begin{proof} Locally about $s=\bar s$, the closed
loop system is equivalent to the three-dimensional dynamical system
\begin{displaymath}
  \begin{cases}
    \begin{aligned}
        \frac{\d}{\d t} s\ & = -\mu(s)b+(\bar D -G_{1}(s-\bar s))(s_\mathrm{in}-s)\\
        \frac{\d}{\d t} b\ & = \phantom{-}\mu(s)b-(\bar D -G_{1}(s-\bar s))b\\
        \frac{\d}{\d t}\bar D & =  -G_{2}(s-\bar s)(\bar
        D-D_{\min})(D_{\max}-\bar D)
  \end{aligned}
\end{cases}
\end{displaymath}
This system admits the unique positive equilibrium $E^{\star}=(\bar s, s_\mathrm{in}-\bar s,\mu(\bar s))$.\\ 
For simplicity, we write the dynamics in the variables $(z,s,\bar D)$ coordinates, where $z$ is defined as $z=s+b$:
\begin{displaymath}
  \begin{cases}
    \begin{aligned}
      \frac{\d}{\d t} z\ & = (\bar D -G_{1}(s-\bar s))(s_\mathrm{in}-z)\\
     \frac{\d}{\d t}  s\ & = -\mu(s)(z-s)+(\bar D -G_{1}(s-\bar s))(s_\mathrm{in}-s)\\
      \frac{\d}{\d t}\bar D  &=   -G_{2}(s-\bar s)(\bar D-D_{\min})(D_{\max}-\bar D)
    \end{aligned}
  \end{cases}
\end{displaymath}
The Jacobian matrix at $E^{\star}$ in these coordinates is
\[
\left(\begin{array}{ccc}
-\mu(\bar s) & 0 & 0\\
 -\mu(\bar s) & -(\mu^{\prime}(\bar s)+G_{1})(s_\mathrm{in}-\bar s) & s_\mathrm{in}-\bar s\\
0 & -G_{2}(\mu(\bar s)-D_{\min})(D_{\max}-\mu(\bar s)) & 0
\end{array}\right)
\]
Its eigenvalues are $\lambda_{1}=-\mu(\bar s)<0$ and $\lambda_{2}$, $\lambda_{3}$ as eigenvalues of the sub-matrix
\[
M=\left(\begin{array}{cc}
-(\mu^{\prime}(\bar s)+G_{1})(s_\mathrm{in}-\bar s) & s_\mathrm{in}-\bar s\\
-G_{2}(\mu(\bar s)-D_{\min})(D_{\max}-\mu(\bar s)) & 0
\end{array}\right)
\]
Then, one has
\[
\begin{array}{lll}
\mbox{det}(M) & = & G_{2}(\mu(\bar s)-D_{\min})(D_{\max}-\mu(\bar s))(s_\mathrm{in}-\bar s )\\
\mbox{tr}(M) & = & -(\mu^{\prime}(\bar s)+G_{1})(s_\mathrm{in}-\bar s)
\end{array}
\]
and concludes about the exponential stability of $E^{\star}$ when $G_{2}>0$ and $G_{1}>-\mu^{\prime}(\bar s)$.
Finally, one obtains from \eqref{dyn-feedback} that $D$ or $\bar D$
converges toward the unknown value $\mu(\bar s)$.\qed
\end{proof}

\begin{figure}[t]
\includegraphics[height=4cm]{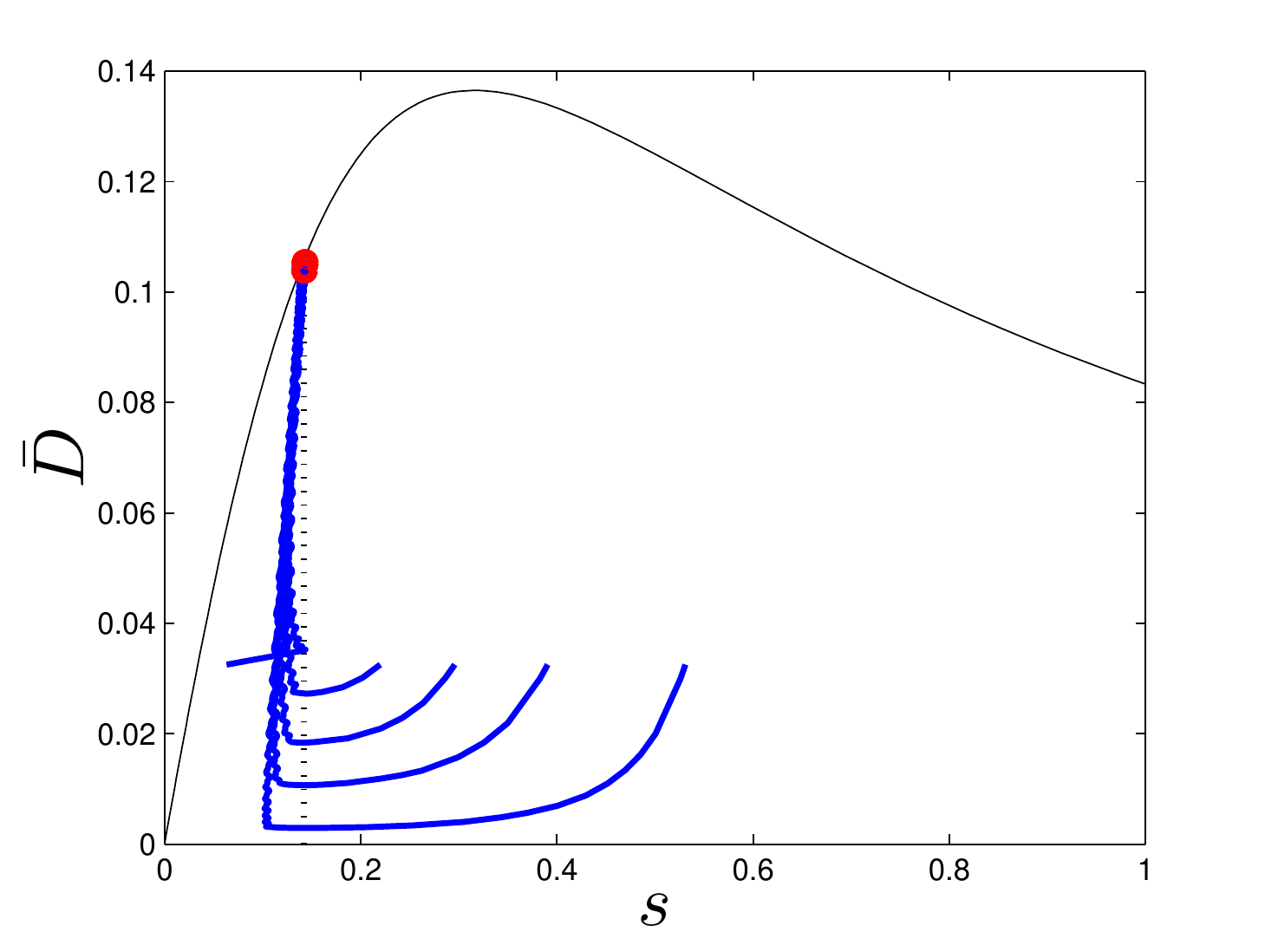}
\includegraphics[height=4cm]{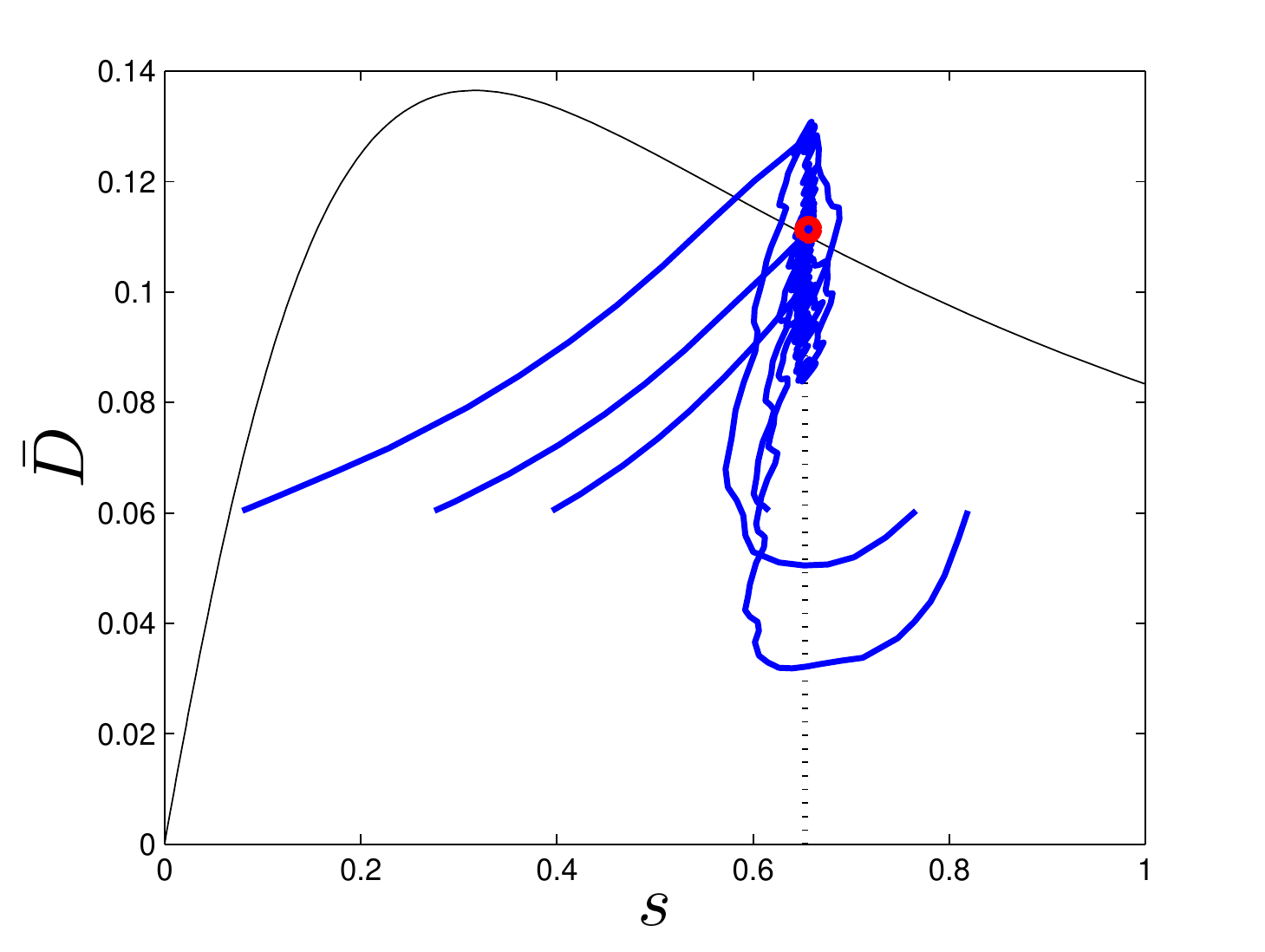}\\
\caption{\label{simu1} The behavior of feedback law
  \eqref{dyn-feedback} for two different values of $\bar s$ in the
  $(s,\bar D)$-projection. Parameters: $G_{1}=2$, $G_{2}=2$, output
  disturbance defined in \eqref{eq:disturbance}.}
\end{figure}

Note that the assumptions in Proposition \ref{proposition2} (for example, on the gain
$G_1$) are weaker than those of Proposition \ref{proposition1} as
Proposition \ref{proposition2}
is only concerned with local stability and a single reference value $\bar s$.

Figure \ref{simu1} demonstrates how control law \eqref{dyn-feedback}
stabilizes an equilibrium with output $\bar s$ for $\bar s$ in the
increasing (left panel) and decreasing (right panel) part of the
growth law $\mu$. For our single-species demonstration we choose the
non-monotonic Haldane function
\begin{equation}
\mu(s)=\frac{s}{1+s+10s^{2}}\label{eq:munonmon}
\end{equation}
and $s_\mathrm{in}=1$. \ar{Any other growth function could have been
chosen, under the requirements that it is Lipschitz continuous and
fulfill equations (\ref{eq:muboundlow}) and (\ref{eq:muboundup}).} 
To illustrate the effect of disturbances, we
super-impose a rapid oscillation onto the measurements of output $s$,
such that the output has the form
\begin{equation}
  \label{eq:disturbance}
  s_\mathrm{output}=s[1+\delta\cos(3t)\sin(t)]\mbox{,\quad where\quad}
  \delta=0.05
\end{equation}
\js{(other disturbances such as quasi-periodic or white-noise signals
  have been tested, getting similar results).}  The grey background
curve in Figure~\ref{simu1} shows $\mu(\cdot)$, which is clearly
non-monotonic on the domain $(0,s_\mathrm{in})$
($s_\mathrm{in}=1$). To show the robustness of the method, we we have
chosen $\bar D(0)$ slightly far from $\mu(s(0))$.

\section{Reconstruction of the growth function}
\label{section_continuation}

Now, we can trace out \js{any desired part of the graph} $\mu(\cdot)$
dynamically by letting $\bar s$ change slowly with time as solution of
the simple dynamics
\begin{equation}\label{eq:sdrift}
\frac{\d}{\d t}\bar s = \varepsilon \bar s(s_\mathrm{in}-\bar s)
\end{equation}
to explore the right part of the graph of $\mu(\cdot)$ when
$\varepsilon$ is a small non-negative number, and to explore the left
one when $\varepsilon$ is a small non-positive number.  During the
reconstruction of the graph, the gain $G_{1}$ has to be been chosen
uniformly large according to \eqref{eq:g1boundmuprime}. Figure
\ref{simu2} shows how the adaptation rule \eqref{eq:sdrift} together
with \eqref{dyn-feedback} reconstructs the entire graph of the growth
function. \js{The overall time it took to reach $\bar s=0.9$ is
  $t\approx3000$ in the dimensionless time units of \eqref{chemostat}.}
\begin{figure}[ht]
\includegraphics[height=4cm]{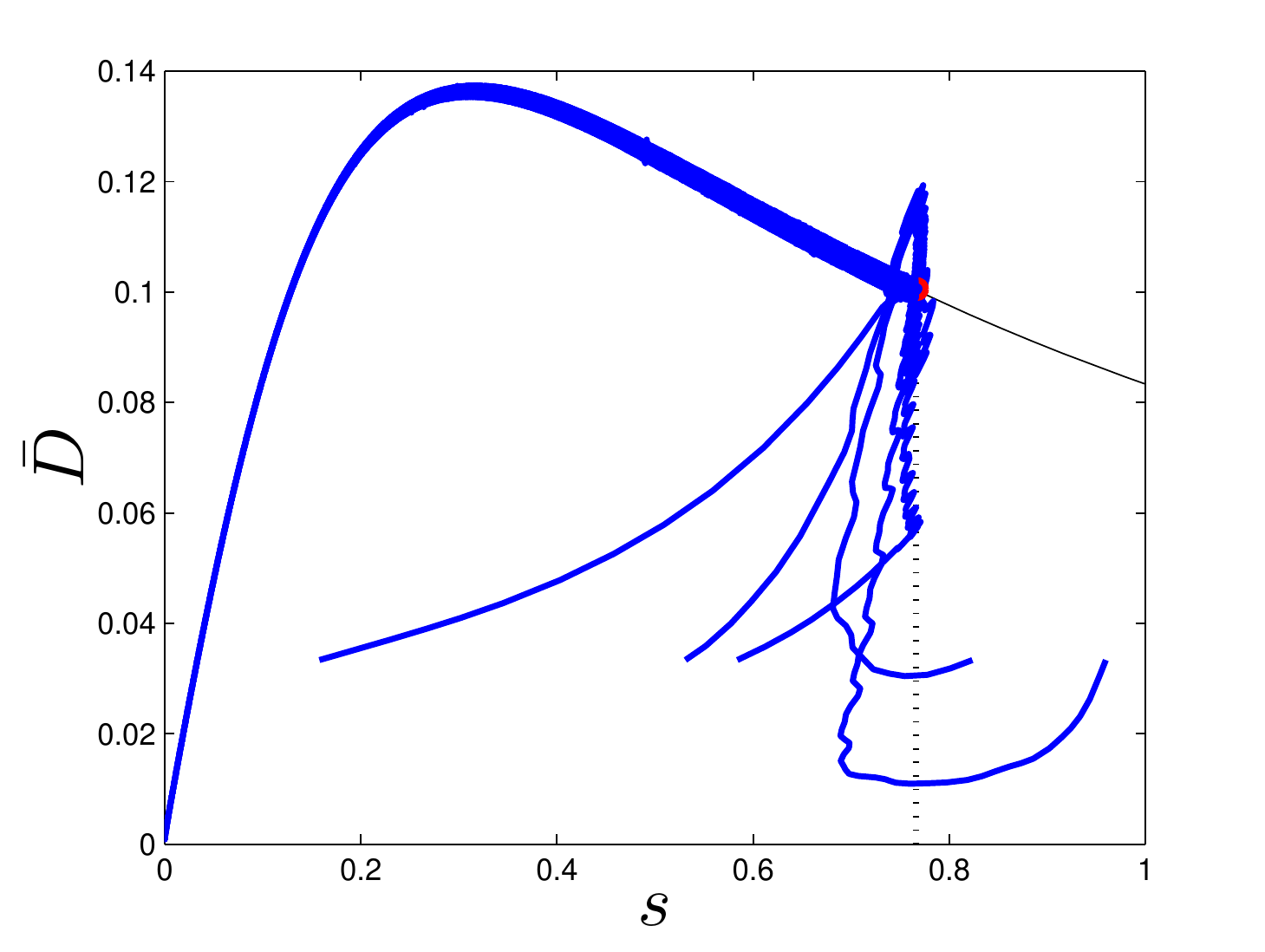}
\includegraphics[height=4cm]{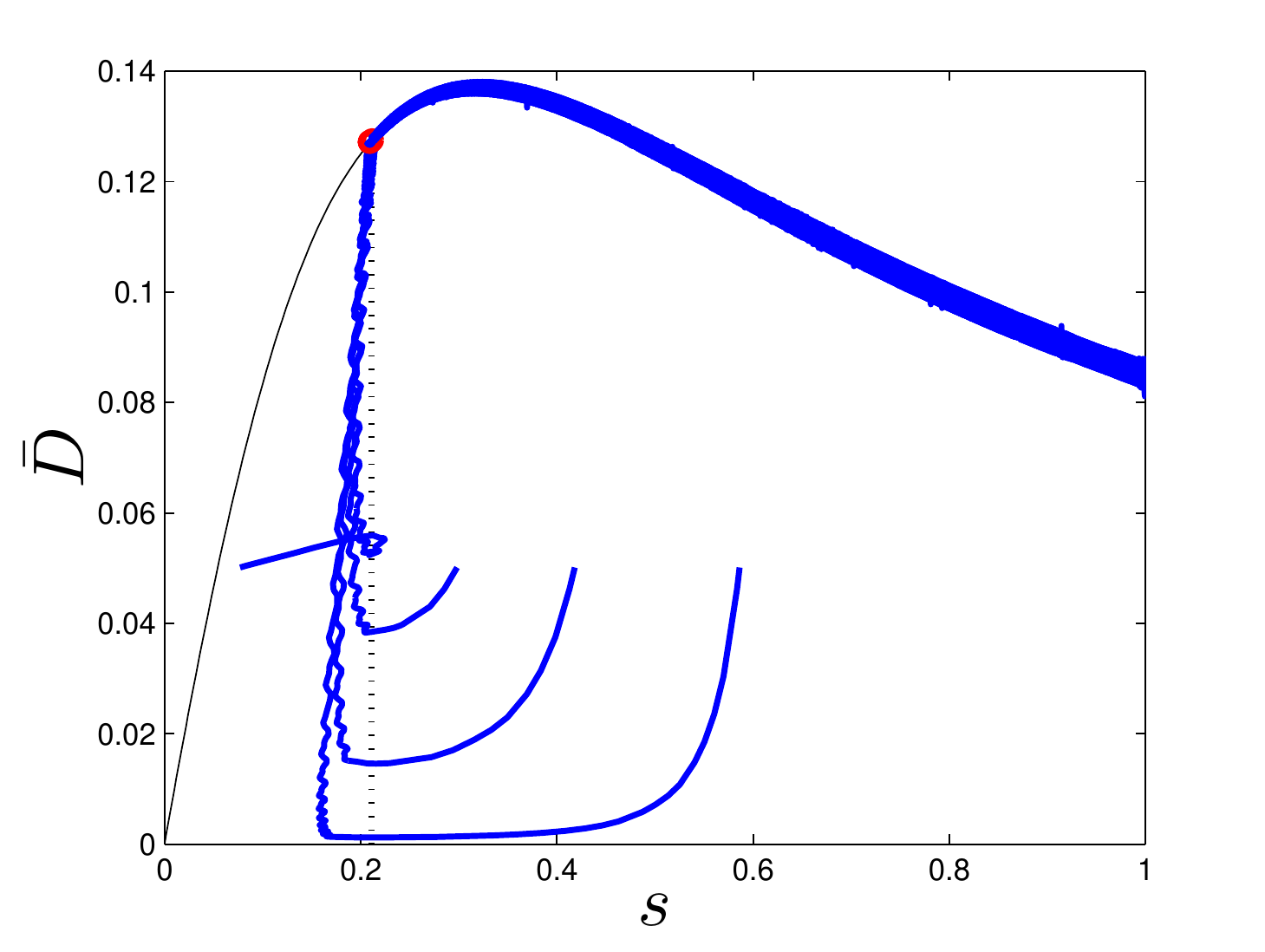}\\
\caption{\label{simu2} Dynamical adaptation using \eqref{eq:sdrift} to
  explore dynamically the left and right part of the graph of the
  unknown function $\mu(\cdot)$. Parameters: for the left panel,
  \js{$\varepsilon=+0.001$}, for the right panel, \js{$\varepsilon=-0.001$},
  otherwise identical to Fig.~\ref{simu1}.}
\end{figure}

\section{Step-wise adaptation of the reference values}
\label{sec:hybrid}
In this section, we propose an alternative to the continuous
adaptation of $\bar D$ and $\bar s$: we treat the root problem
$0=s_\mathrm{eq}(\bar s,\bar D)-\bar s$ with ordinary numerical
root-finders such as the Newton iteration. We present here an approach
that combines the two steps of the method (the adaptive control and
the continuation) in a step-wise framework.

\subsection{Adaptation using Newton iteration}
\label{sec:discnewton}

In an experimental setting one will have to adapt the numerical
methods to the lower accuracy of experimental outputs (see
\cite{SGNWK08} for a demonstration in a mechanical experiment) but for
this paper we restrict ourselves to a numerical demonstration. In the
single-species chemostat one profits from the knowledge of an
approximate derivative of $s_\mathrm{eq}$ with respect to $\bar D$,
making the Newton iteration more efficient. Suppose, we plan to
identify the growth function $\mu$ in a sequence of points $\bar
s_k=\bar s_0+k\delta$ (where $\delta>0$ is small). The function values
$\mu(\bar s_k)$ are the roots $\bar D_k$ of $s_\mathrm{eq}(\bar
s_k,\cdot)-\bar s_k$, \js{where $s_\mathrm{eq}(\bar s,\bar D)$ was the
  asymptotic output of the chemostat \eqref{chemostat} with simple
  feedback control \eqref{simple-feedback}, as defined by
  \eqref{eq:seqdef}. The equilibrium value $s_\mathrm{eq}(\bar s,\bar
  D)$ satisfies $\mu(s_\mathrm{eq}(\bar s,\bar
  D))=D-G_1(s_\mathrm{eq}(\bar s,\bar D)-\bar s)$ due to
  \eqref{chemostat} (see also \eqref{eq:seqdef:imp}) for all
  admissible $\bar D$. Differentiating this implicit expression with
  respect to $\bar D$, we obtain
\begin{displaymath}
  \frac{\partial s_\mathrm{eq}}{\partial \bar D}(\bar D,\bar s_k)=
  \frac{1}{G_1+\mu'(s_\mathrm{eq}(\bar D,\bar s_k))}\approx
  \frac{1}{G_1+\frac{\bar D-\bar D_{k-1}}{\bar s_k-\bar s_{k-1}}}\mbox{,}
\end{displaymath}
where we used a secant approximation for $\mu'$ on the right-hand side.}
This leads to the iteration rule
\begin{equation}\label{eq:newton}
  \bar D_\mathrm{new}=\bar D_\mathrm{old}-
  \frac{s_\mathrm{eq}(\bar D_\mathrm{old},\bar s_k)}{%
    G_1+\frac{\bar D_\mathrm{old}-\bar D_{k-1}}{\bar s_k-\bar s_{k-1}}}\mbox{,}
\end{equation}
starting from $\bar D=\bar D_{k-1}$, or (for $k>2$)
\begin{displaymath}
  \bar D=\bar D_{k-1}+
  \frac{\bar D_{k-1}-\bar D_{k-2}}{\bar s_{k-1}-\bar s_{k-2}}\mbox{.}
\end{displaymath}
For the initial step ($k=1$) the derivative of $s_\mathrm{eq}$ has to
be either guessed or approximated with a finite difference (we used
the latter in our numerical simulations).

Note that at no point it is necessary to set the internal states $s$ or $b$
of system~\eqref{chemostat}. Only the reference values $(\bar s,\bar
D)$ have to be set.
\begin{figure}
  \includegraphics[width=0.9\textwidth]{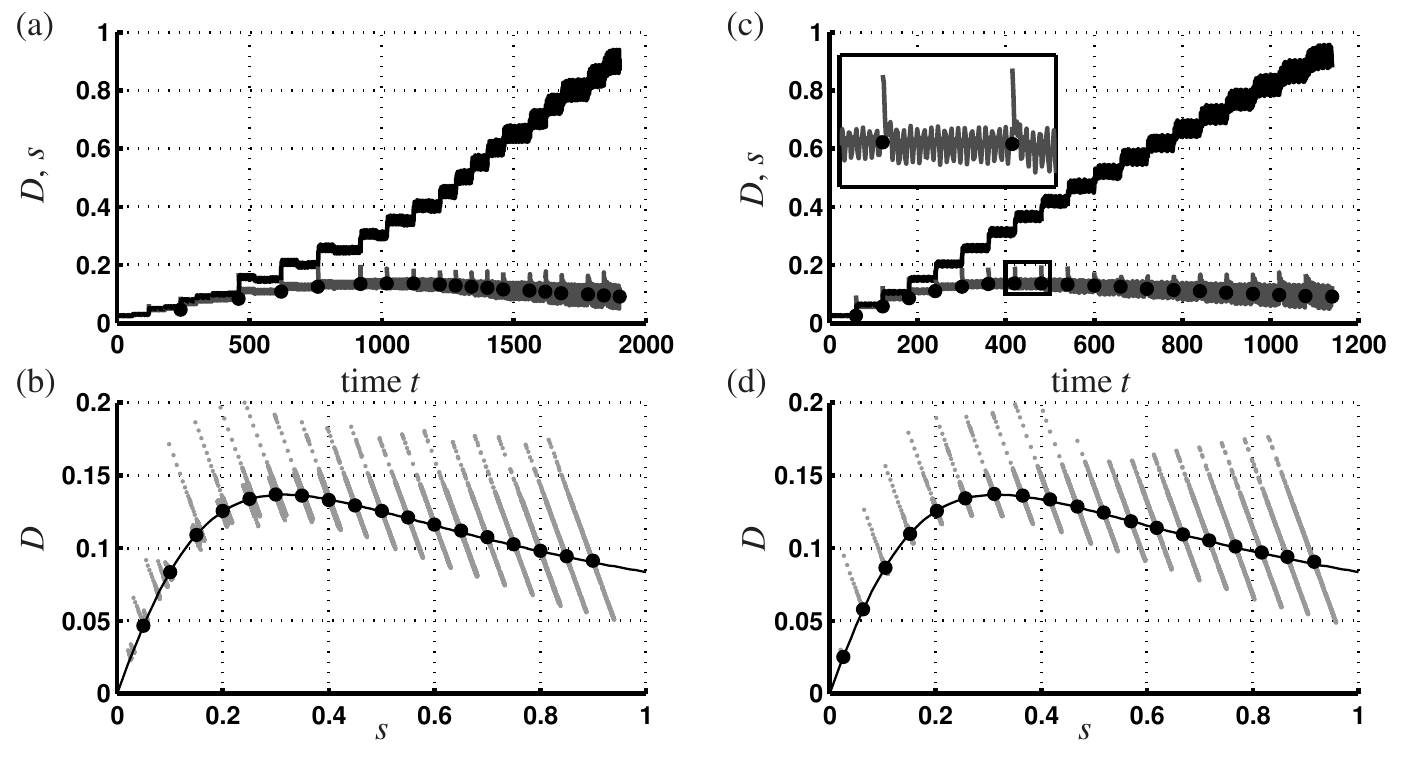}
  \caption{ Simulation of step-wise adaptation. Panels (a) and (b) use
    \eqref{eq:newton} where $\bar s_k=0.05k$. Panels (c) and (d) use
    \eqref{eq:simplecont}--\eqref{eq:simplecontnext} where
    $\delta=0.05$, both cases subject to output
    disturbance~\eqref{eq:disturbance}. Parameters:
    $\mathtt{tol}=10^{-3}$ (for (a) and (b)), $G_1=1$, $D_{\max}=0.2$,
    $D_{\min}=0.02$. A check if the transients have settled was
    performed every $20$ time units. Transients were accepted as
    settled if the standard deviation of $s$ ($\std(s)$) on the last
    interval is no longer smaller than $0.9\std(s)$ on the previous
    interval. The average over the last interval is used as the
    resulting equilibrium.}
  \label{fig:hybrid}
\end{figure}
The panels (a) and (b) of Figure~\ref{fig:hybrid} show the output of a
simulation with the step-wise adaptation using Newton iteration
\eqref{eq:newton}.  Panels (a) shows the time profile of output $s$
and input $D$ throughout the run. Panels (b) shows the evolution in
the $(s,D)$-plane in grey. Black dots indicate when convergence was
reached ($|s_\mathrm{eq}-\bar s_k|<\mathtt{tol}$. These points
correspond to values at which the control was accepted as
non-invasive. Then the iteration moved on to the next $\bar s_k$.  By
gradually tracing out the graph of $\mu$, one achieves small and
rapidly decaying transients in every evaluation of $s_\mathrm{eq}$
(which involves running system~\eqref{chemostat} with control until
transients have settled). This is so because the transients all lie
inside the subspace $\{(s,b): s+b=s_\mathrm{in}\}$ after
system~\eqref{chemostat} has run at least once. Second, the initial
offset from the equilibrium is always small, because the adjustments
of $\bar s$ and $\bar D$ are small.

\subsection{A simplified step-wise scheme}
\label{sec:discsimple}
The scheme~\eqref{eq:newton} permits one to find $\mu(\bar s_k)$ for
an a priori prescribed set of admissible abscissae $\bar s_k$. If one
wants to recover only the graph of $\mu$ one does not need to
prescribe the sequence $\bar s_k$ a priori, thus, avoiding a Newton
iteration. Suppose that we know already two points $p_{k-1}=(\bar
s_{k-1}, \bar D_{k-1})$ and $p_k=(\bar s_k, \bar D_k)$ on the curve
$(s,\mu(s))$. Then we set
\begin{equation}
  \label{eq:simplecont}
  (\bar s_{\mathrm{pred},k+1},\bar D_{\mathrm{pred},k+1})=
  p_k+\delta \frac{p_k-p_{k-1}}{%
    \|p_k-p_{k-1}\|}\mbox{,}
\end{equation}
where $\delta>0$ is the approximate desired distance between points
along the curve $(s,\mu(s))$, and run the controlled experiment with
the reference values $(\bar s,\bar D)=(\bar s_{\mathrm{pred},k+1},\bar
D_{\mathrm{pred},k+1})$ in \eqref{simple-feedback} until the
transients have settled to obtain the next point on the curve
\begin{equation}
  \label{eq:simplecontnext}
  \begin{split}
    \bar s_{k+1}&=s_\mathrm{eq}(\bar s_{\mathrm{pred},k+1},
    \bar D_{\mathrm{pred},k+1})\\
    \bar D_{k+1}&=D(\bar s_{k+1},\bar D_{\mathrm{pred},k+1},
    \bar s_{\mathrm{pred},k+1})\\
    &=\bar D_{\mathrm{pred},k+1}-G_1(\bar s_{k+1}-\bar s_{\mathrm{pred},k+1})
  \end{split}
\end{equation}
This simplified procedure cannot guarantee the identification of $\mu$
at prescribed equidistantly spaced values of $s$ but finds $\mu(\bar
s_k)$ for a (nearly evenly spaced) sequence $\bar s_k$ given by the
intersections of the lines $D=D_{\mathrm{pred},k}-G_1(s-\bar
s_{\mathrm{pred},k})$ with the graph $D=\mu(s)$.

Figure~\ref{fig:hybrid}, panels (c) and (d), demonstrate the speed-up
using the simplified scheme
\eqref{eq:simplecont}--\eqref{eq:simplecontnext} (note the times at
the abscissae). The difference to Figure~\ref{fig:hybrid}(a,b) is that
the values $\bar s_k$ at which the growth function is evaluated are
not exactly equidistantly spaced. \js{The zoom in
  Figure~\ref{fig:hybrid}(c) shows that the control reaches the
  equilibrium up to an error at the level of the disturbance very
  quickly. The black dot shows then the average of the output during
  the remainder of the time before the output gets accepted (thus,
  achieving higher accuracy at the cost of speed).}

\section{The two species case}
\label{sec:twospecies}
Let us now consider an extension of the chemostat model
\eqref{chemostat} that considers two species which compete for the
same substrate. The two-species model can be written as follows
\begin{equation}
\label{chemostat2}
\left\{\begin{array}{lll}
\dot s & = & \ds -\sum_{i=1}^{2}\mu_{i}(s)b_{i}+D(s_\mathrm{in}-s)\\[3mm]
\dot b_{i} & = & \mu_{i}(s)b_{i}-Db_{i} \qquad (i=1,2)
\end{array}\right.
\end{equation}
The two-species model has co-existing equilibria $E^*_i$, which
correspond to the state where species $i$ is present and the other
species $3-i$ is suppressed. The following proposition shows first
that feedback stabilization based on input $D$ and output $s$ breaks
down in general for the equilibrium corresponding to the species with
the smaller growth rate (the \emph{suppressed}, or
\emph{non-dominant}, species). Then we state what eigenvalues the
linearizations at equilibria have for our specific control laws,
\eqref{simple-feedback} and \eqref{dyn-feedback}.
\begin{proposition}
\label{proposition3}
Fix $\bar s \in (0,s_\mathrm{in})$ and consider the equilibrium $E_{2}^{\star}=(\bar s,0,s_\mathrm{in}-\bar s)$.
\begin{enumerate}
\item \label{thm:genfail} \emph{\textbf{(Suppressed equilibrium not
      stabilizable)}} Let $\mu_{1}(\bar s)> \mu_{2}(\bar s)$, and
  $D(\cdot)$ be a feedback $D(\cdot)$ of the form
  \begin{equation}
    D=f(s,\xi)\mbox{,} \quad \dot \xi=g(s,\xi)\mbox{, 
      \qquad( $\xi \in \mathbb{R}^{k}$)}\label{eq:genfb}
\end{equation}
with $f(\bar s,0)=\mu_{2}(\bar s)$ and $g(\bar s,0)=0$. Then the
equilibrium $E_{2}^{\star}$ of system \eqref{chemostat2} with feedback
$D$ is unstable.
\item\label{thm:simpleev} \emph{\textbf{(Eigenvalues for simple
      feedback)}} Using feedback law \eqref{simple-feedback},
  \begin{displaymath}
    D(s,\bar D,\bar s)=\sat_{[D_{\min},D_{\max}]} \left(\bar D -
      G_{1}(s-\bar s)\right)
  \end{displaymath}
  with $D_{\min}$ and $D_{\max}$ such that
  $D_{\min}<\mu_2(\bar s)<D_{\max}$, the linearization of system
  \eqref{chemostat} in the equilibrium $E_2^\star$ has the eigenvalues
  $-\mu_2(\bar s)$, $\mu_1(\bar s)-\mu_2(\bar s)$ and $-(\mu'_2(\bar
  s)+G_1)(s_\mathrm{in}-\bar s)$.
\item \label{thm:dynev} \emph{\textbf{(Dominant equilibrium stabilized
      by dynamic feedback)}} If $\mu_{1}(\bar s)< \mu_{2}(\bar s)$,
  then the feedback \eqref{dyn-feedback} exponentially stabilizes the
  system \eqref{chemostat2} locally about $E_{2}^{\star}$, for any
  positive constants $(G_{1},G_{2})$ such that $G_{1}>
  -\mu_{2}^{\prime}(\bar s)$.  Furthermore one has
\[
\lim_{t\to+\infty}\bar D(t)=\mu_{2}(\bar s)
\]
\end{enumerate}
\end{proposition}

\begin{proof}
  Consider the dynamics of the two-species model \eqref{chemostat2}
  with feedback $D$ given by \eqref{eq:genfb}
\[
\left\{\begin{array}{lll}
\dot s & = & \ds -\sum_{i=1}^{2}\mu_{i}(s)b_{i}+f(s,\xi)(s_\mathrm{in}-s)\\[4mm]
\dot b_{i} & = & \mu_{i}(s)b_{i}-f(s,\xi)b_{i} \qquad (i=1,2)\\
\dot \xi & = & g(s,\xi)\mbox{.}
\end{array}\right.
\]
We write this system in $(z,b_{1},b_{2},\xi)$ coordinates with
$z=s+b_{1}+b_{2}$:
\[
\left\{\begin{array}{lll}
\dot z & = & f(z-b_{1}-b_{2},\xi)(s_\mathrm{in}-z)\\
\dot b_{i} & = & (\mu_{i}(z-b_{1}-b_{2})-f(z-b_{1}-b_{2},\xi))b_{i}\\
\dot \xi  & = & g(z-b_{1}-b_{2},\xi)\mbox{.}
\end{array}\right.
\]

Point~\ref{thm:genfail}: at equilibrium $E_{2}^{\star}$, the Jacobian
matrix $J_2^{\star}$ possesses the following form in
$(z,b_{1},b_{2},\xi)$ coordinates
\begin{equation}\label{eq:gen2sev}
  J_2^\star=\left(\begin{array}{cccc}
      -\mu_{2}(\bar s) & 0 & 0 & 0\\
      0 & \mu_{1}(\bar s)-\mu_{2}(\bar s) & 0 & 0\\
      \star & \star & \star & \star\\
      \star & \star & \star & \star
\end{array}\right)\mbox{,}
\end{equation}
which has the positive eigenvalue $\mu_{1}(\bar s)-\mu_{2}(\bar s)$. This proves that $E^*_2$ is unstable whatever the choice of the feedback $D(\cdot)$.\\

\noindent
Point~\ref{thm:simpleev}: we can be more specific about the form of
$J_2^\star$ for the simple feedback law~\eqref{simple-feedback}. Since
$D_{\min}<\mu_2(\bar s)<D_{\max}$, the feedback is in its linear
regime, such that $\partial_sD=-G_1$. Thus, (component $\xi$ is
absent)
\begin{displaymath}
  J_2^\star=
  \begin{pmatrix}
    -\mu_{2}(\bar s) & 0 & 0 \\
    0 & \mu_{1}(\bar s)-\mu_{2}(\bar s) & 0\\
    \star &\star & -(\mu'_2(\bar
    s)+G_1)(s_\mathrm{in}-\bar s)
  \end{pmatrix}
\end{displaymath}

\noindent
Point~\ref{thm:dynev}: for feedback law \eqref{dyn-feedback}
$J_2^\star$ can be written as follows, in $(z,b_{1},b_{2},\bar D)$
coordinates
\begin{displaymath}
J_2^\star=
\begin{pmatrix}
  \begin{matrix}
    -\mu_{2}(\bar s) & 0 \\
    0 &\mu_{1}(\bar s)-\mu_{2}(\bar s)\\ 
    \star &\star    
  \end{matrix}
  \quad&  
  \begin{matrix}
    \begin{matrix}
      0&0\\ 0&0
    \end{matrix}\\
    M
\end{matrix}
\end{pmatrix}\mbox{,}
\end{displaymath}
where $M$ is a $2\times2$ matrix with the entries
\[
M=
\begin{pmatrix}
  -(\mu_{2}^{\prime}(\bar s)+G_{1})(s_\mathrm{in}-\bar s) &\quad& 
  s_\mathrm{in}-\bar s\\
  -G_{2}(\mu_{2}(\bar s)-D_{\min})(D_{\max}-\mu_{2}(\bar s)) &\quad& 0
\end{pmatrix}
\]
Its eigenvalues are $\lambda_{1}=-\mu_{2}(\bar s)<0$, $\lambda_{2}=\mu_{1}(\bar s)-\mu_{2}(\bar s)<0$, $\lambda_{3}$ and $\lambda_{4}$ with
\[
\begin{array}{lll}
\lambda_{3}\lambda_{4} & = & G_{2}(\mu_{2}(\bar s)-D_{\min})(D_{\max}-\mu_{2}(\bar s))(s_\mathrm{in}-\bar s )\\
\lambda_{3} + \lambda_{4} & = & -(\mu_{2}^{\prime}(\bar s)+G_{1})(s_\mathrm{in}-\bar s)
\end{array}
\]
As in the proof of Proposition \ref{proposition2}, one concludes the exponential
stability of $E^{\star}_{2}$ when $G_{2}>0$ and
$G_{1}>-\mu_{2}^{\prime}(\bar s)$, and the convergence of $D(\cdot)$
toward $\mu_{2}(\bar s)$.\qed
\end{proof}

Consequently, the adaptive control scheme proposed in Section
\ref{section-adaptive} only allows one to reconstruct the larger of
the two growth rates at any given $s$ by stabilizing the equilibrium. 

Nevertheless, the introduction of feedback control may still be of
help. To be specific, let us assume that species $1$ is dominant for
$s<s_c$ (and species $2$ is suppressed there), and species $2$ is
dominant for $s>s_c$, where $s_c$ is a cross-over point:
$\mu_1(s)>\mu_2(s)$ for $s<s_c$ and $\mu_1(s)<\mu_2(s)$ for $s>s_c$
(see the underlying function graphs in Figure~\ref{simu2species} for a
typical picture of the discussed scenario, $s_c=0.5$ in
Fig.~\ref{simu2species}). Suppose we are interested in the location of
$E_2^\star$ to identify $\mu_2(\bar s)$ for $\bar s<s_c$. As the form
of the Jacobian $J_2^\star$ in \eqref{eq:gen2sev} makes clear, two
eigenvalues of $J_2^*$ are unaffected by our feedback control. One of
them, $-\mu_2(\bar s)$ is always stable. It corresponds to the
transversally stable direction of the invariant subspace
$T=\{(s,b_1,b_2):s+b_1+b_2=s_\mathrm{in}\}$. Once, the system is in
$T$, control will not move it out of $T$, thus, we can ignore this
eigenvalue.
  
The other uncontrollable eigenvalue, $\mu_1(\bar s)-\mu_2(\bar s)$, is
unstable for $\bar s<s_c$, but stable for $\bar s>s_c$. Consequently,
the following strategy would make it possible in principle to identify
$\mu_2$ for $\bar s<s_c$: keep the system in the region where $s>s_c$
for some time to suppress species $1$ (which will exponentially decay
for $s>s_c$ according to Proposition~\ref{proposition3}, point
\ref{thm:simpleev}). When one is sufficiently close to the invariant
line $L_2=\{(s,b_1,b_2): b_1=0, s+b_2=s_\mathrm{in}\}$, one is
(nearly) in the single-species case, where one can then use the
methods of Sections~\ref{section-adaptive}, \ref{section_continuation}
or \ref{sec:hybrid} to explore $\mu_2$ for a finite time for $s<s_c$
until species $1$ has recovered. This approach is also possible
without control if both growth functions are monotone.

Figure~\ref{simu2species}(b) demonstrates the application of feedback
law \eqref{dyn-feedback} in combination with \eqref{eq:sdrift}, which
defines the adaptive control presented in Section
\ref{section_continuation}, to the two-species situation with the
monotonic Monod functions 
\begin{equation}
  \label{eq:mutwo}
  \mu_{1}(s)=\frac{s}{0.1+s}, \qquad \mu_{2}(s)=1.5\frac{s}{0.4+s}
\end{equation}
as growth rates.

If one treats $\bar s$ as a parameter then the equilibria $E^*_1$ and
$E^*_2$ undergo an exchange of stability (a degenerate transcritical
bifurcation) at $\bar s=s_c$. As the adaptation rule \eqref{eq:sdrift}
lets $\bar s$ drift slowly (with speed $\varepsilon$) the full system
exhibits a phenomenon known as delayed loss of stability
\cite{BS99,LRS09} in the context of dynamic bifurcations \cite{B91},
widely studied in slow-fast systems \cite{O91}. \js{Say, we are
  decreasing $\bar s$ slowly from above $s_c$ to below $s_c$ (as in
  Figure~\ref{simu2species}(a)). Then concentration $b_1$ decreases
  exponentially, coming close to $0$ while $\bar s>s_c$.} After $\bar
s$ has crossed $s_c$, the concentration $b_1$ grows exponentially, but
still takes some time until it reaches values noticeably different
from $0$. The value $\bar s_\mathrm{loss}$ of $\bar s$ at which $b_1$
becomes noticeably non-zero is in the ideal ODE model independent of
the drift speed $\varepsilon$ of $\bar
s$. Figure~\ref{simu2species}(a) shows this effect: since $b_1$ is
nearly zero the variable $D$ continues to follow the, by now unstable,
drifting equilibrium $E^*_2$. \js{For a given $\bar s>s_c$ and an
  arbitrary small $b_{1\mathrm{ini}}>0$ at time $0$, the value $\bar
  s_\mathrm{loss}<s_c$ at which $b_1$ reaches $b_{1,\mathrm{ini}}$
  again when following \eqref{eq:sdrift}, is given implicitly by
  the relation }
\[
\int_{\bar s}^{\bar s_\mathrm{loss}} (\mu_1(\sigma)-\mu_2(\sigma))/(\sigma(s_{in}-\sigma))d\sigma=0
\]
in the limit of small $\epsilon$.  This delay mechanism allows one in
principle a reconstruction of a part of the smaller growth rate close
to the bifurcation value of $\bar s$.

\emph{Remark 1}: From a practical view point, an apparent jump between
the two graphs (see evolution curves in black in
Fig.~\ref{simu2species}) could indicate the presence of another
species, if one believes that the culture in chemostat was initially
pure. Nevertheless, one can still rely on the reconstruction of parts
of growth curves for each species.

\begin{figure}[ht]
\includegraphics[width=6cm,height=3cm]{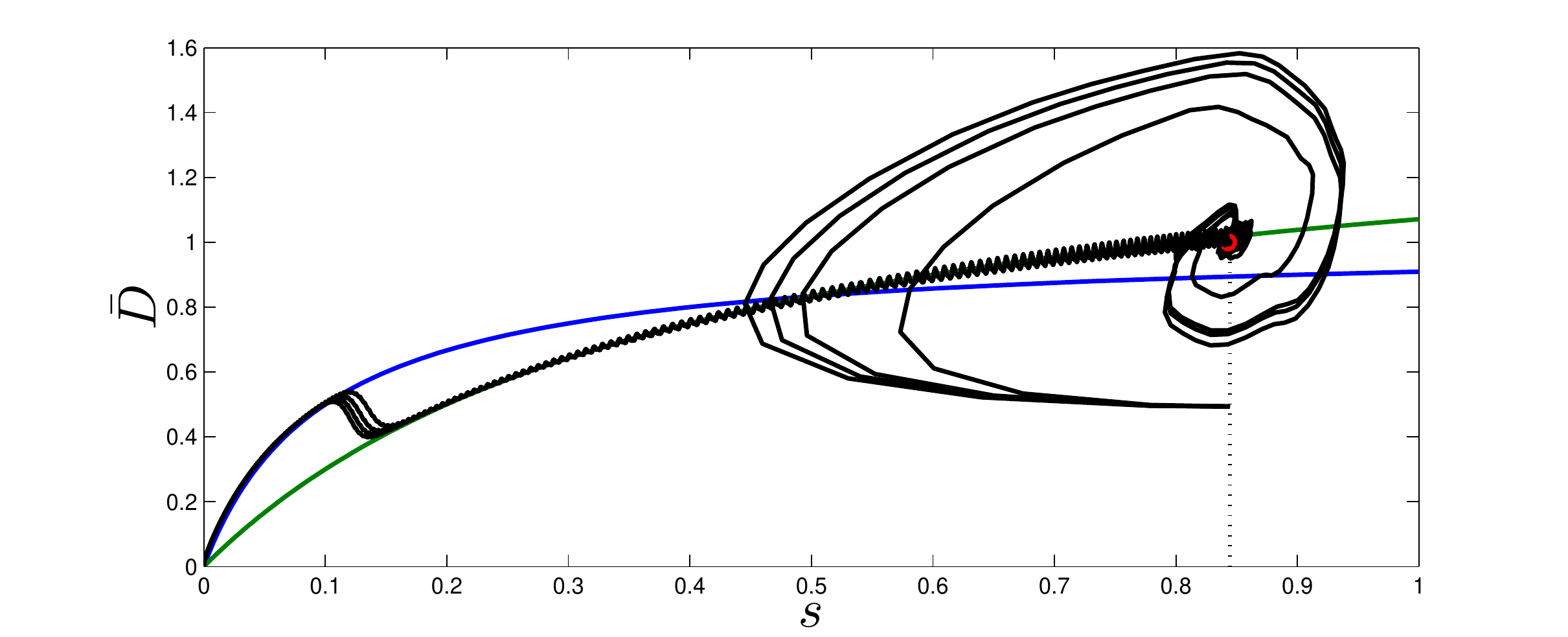}
\includegraphics[width=6cm,height=3cm]{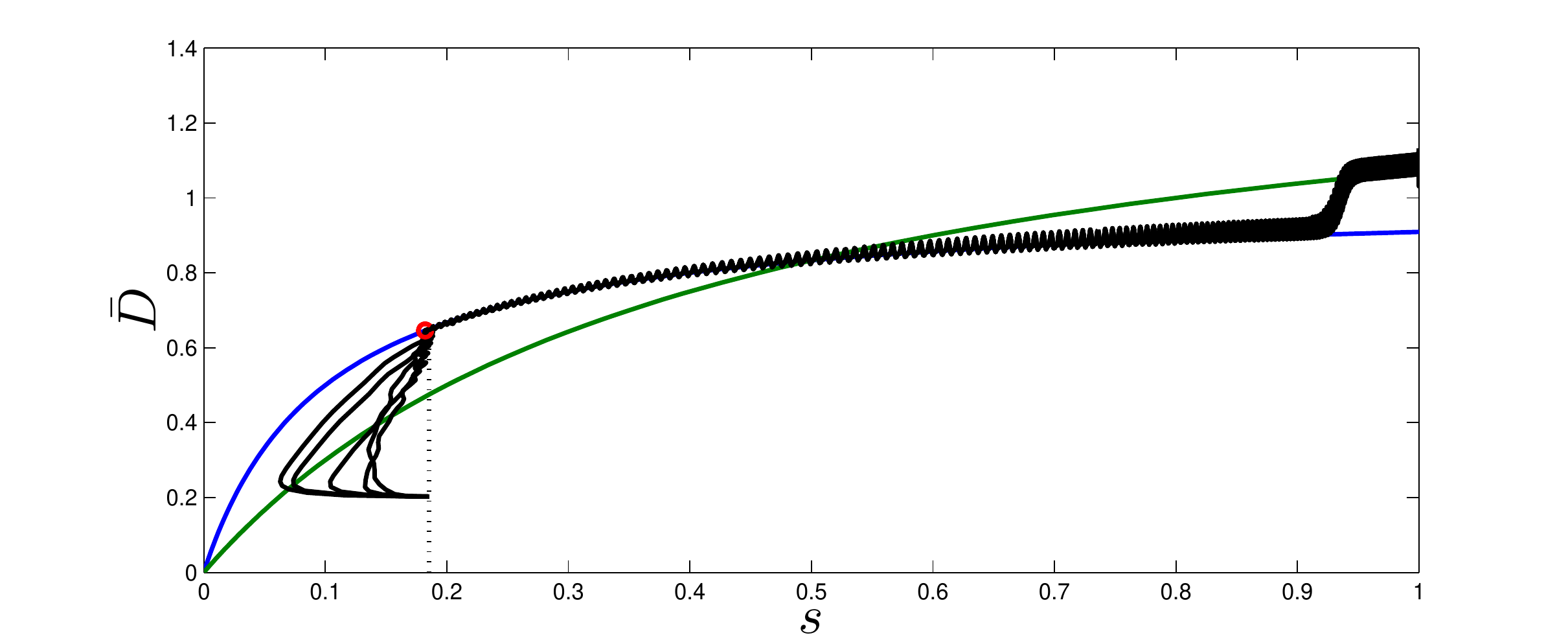}
\caption{\label{simu2species} Exploiting the delayed loss of stability
  for the dynamic control feedback law \eqref{dyn-feedback} in the two
  species case with growth rates given in \eqref{eq:mutwo} (shown
  underlying in both panels). The reference value $\bar s$ is
  \js{decreasing} on the left panel, and \js{increasing} on the right
  one. Parameters: \js{$\epsilon=-0.01$} for panel (a) and \js{$\epsilon=0.01$}
  for panel (b) in \eqref{eq:sdrift}, $G_1=2$, $G_2=2$ in
  \eqref{dyn-feedback}, output disturbance defined in
  \eqref{eq:disturbance}.}
\end{figure}
\begin{figure}
\includegraphics[width=\textwidth]{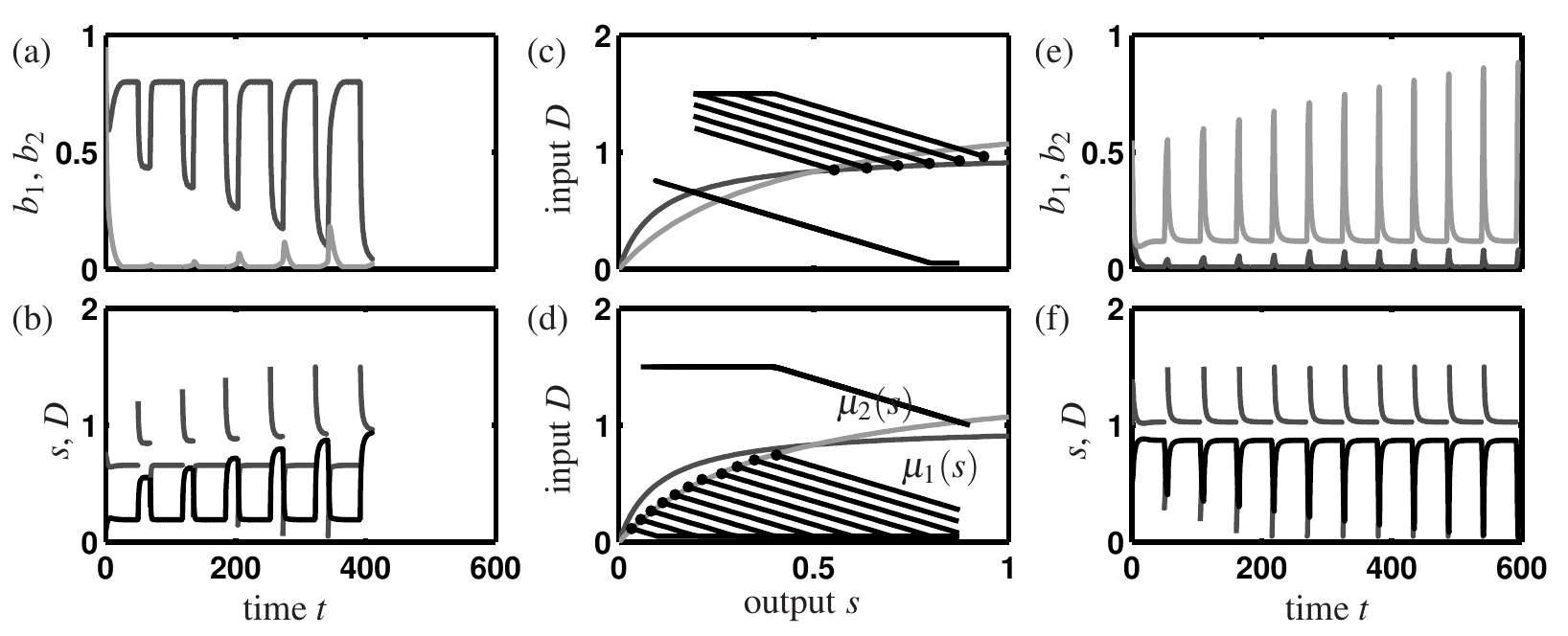}
\caption{\label{discrete2species} 
  Jumping back and forth between two regions at discrete times
  exploiting the delayed loss of stability for the simple control
  feedback law \eqref{simple-feedback} in the two species case with
  growth rates given in \eqref{eq:mutwo} (shown underlying in both
  panels (c) and (d)).}
\end{figure}
\begin{figure}[ht]
  \centering
  \includegraphics[width=0.6\textwidth]{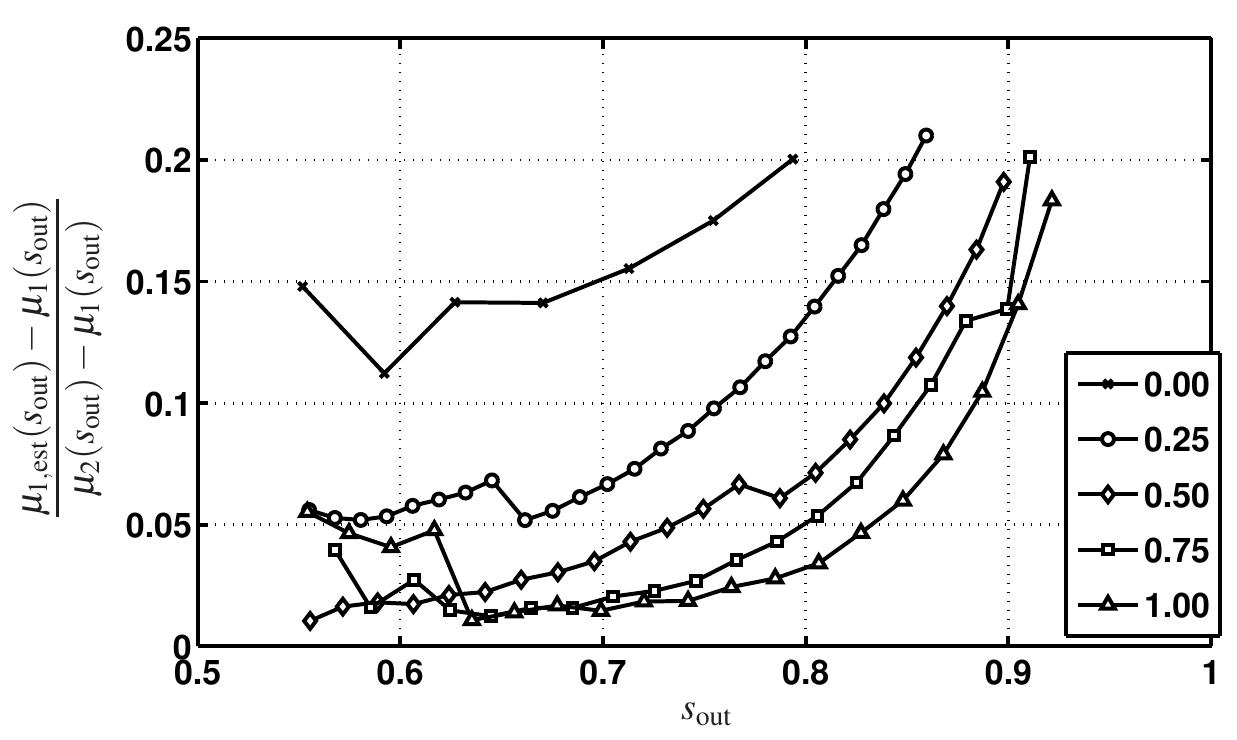}  
  \caption{ Effect of
    feedback control \eqref{simple-feedback} on accuracy of identified
    growth rate $\mu_2$. Parameters are the same as for
    Figure~\ref{discrete2species}(a-c) (in particular,
    $b_{\min}=10^{-3}$), the gain $G_1$ is shown in the
    legend. Parameters: $G_1=1$, $D_{\max}=1.5$, $D_{\min}=0.1$ in
    \eqref{simple-feedback}, $b_{\min}=10^{-3}$, output disturbance
    defined in \eqref{eq:disturbance}.}
  \label{fig:gains2species}
\end{figure}
\emph{Remark 2}: In the idealized ODE model one could in principle
recover the entire suppressed part of the growth rate by spending more
time initially in the dominant part. This is so, because the
concentration of the suppressed species (say, $b_1$ for $s>s_c$ in
Fig.~\ref{simu2species}(b)) can be made arbitrarily small by spending
more time with $\bar s>s_c$. In practice, the suppression of species
may not be perfect. For example, it may be impossible to suppress
either species below a concentration $b_{\min}>0$. Then this
concentration $b_{\min}$ determines for how long the system will stay
close to the invariant plane $\{b_1=0\}$, when this plane is
unstable. Thus, $b_{\min}$ determines how close one can get to the
unstable equilibrium $E^\star_2$ for $\bar s<s_c$ (together with the
difference in growth rates, $\mu_1(\bar s)-\mu_2(\bar s)$, which
determines how unstable the plane $\{b_1=0\}$ is in $\bar s$). This is
where the feedback control \eqref{simple-feedback},
$D=\sat_{[D_{\min},D_{\max}]} \left(\bar D - G_{1}(s-\bar s)\right)$,
has an effect: the unstable equilibrium $E^\star_2$ has stronger
attraction along the invariant line $\{(s,b_1,b_2): b_1=0,
s+b_1+b_2=s_\mathrm{in}\}$, because the third eigenvalue of
$J_2^\star$, $-(\mu'_2(\bar s)+G_1)(s_\mathrm{in}-\bar s)$, can be
made more negative by increasing $G_1$.

Figure~\ref{discrete2species} demonstrates that it is possible in
principle to identify the growth rates of species in ranges of $s$
where they are suppressed \js{(even for positive
  $b_\mathrm{min}=10^{-3}$)}, if the species is dominant in another
region. The procedure was as follows (for
Fig.~\ref{discrete2species}(a-c)):
\begin{enumerate}
\item Set $(\bar s,\bar D)$ to \texttt{(0.1,0.75)}, and wait until
  transients have settled (implying that species $2$ is
  suppressed). The output $s$ settles to a value less than $s_c=0.5$.
\item Then set $(\bar s,\bar D)$ to, say,
  \texttt{(0.9,0.75)}. One expects a transient that initially
  follows the invariant line $\{b_2=0, b_1=s_\mathrm{in}-s\}$ where
  $s$ initially increases.
\item As soon as $s$ stops increasing (let's say, at
  $s_\mathrm{out}$), we know that the system now moves away from the
  plane $\{b_2=0\}$. So, we read off $D$, which is the estimate
  $\mu_{1,\mathrm{est}}(s_\mathrm{out})$ (a black dot in
  Fig.~\ref{discrete2species}(c)), and go back to step 1.
\end{enumerate}
In Figure~\ref{discrete2species}(a-c) switching occurs between $(\bar
s,\bar D)=(\mathtt{0.1,0.75})$ (where species 2 is stable) and $(\bar
s,\bar D)=(\mathtt{0.4:0.1:1,0.9})$ (reading off $\mu_2$). 

Figure~\ref{discrete2species}(d-f) demonstrates the same procedure for
identifying $\mu_2$ for $s<s_c$. The only difference is that we read
off $D=\mu_{2,\mathrm{est}}(s_\mathrm{out})$ at an inflection point
$s_\mathrm{out}$ of $s(t)$. In Figure~\ref{discrete2species}(d-f)
switching occurs between $(\bar s,\bar D)=(\mathtt{0.9,1})$ (where
species $2$ is dominant) and $(\bar s,\bar
D)=(\mathtt{1:-0.1:0,0.15})$ (reading off $\mu_2$ in the region where
species $2$ is suppressed).

The procedure, with its steps 1--3 is also possible without feedback
control \eqref{simple-feedback}. However, feedback control
\eqref{simple-feedback} increases the decay rate of the equilibrium
with respect to disturbances, \js{for example, within the invariant line
$\{b_2=0, b_1=s_\mathrm{in}-s\}$ in
Figure~\ref{discrete2species}(a-c)}. Thus, the system trajectory will
come closer to the equilibrium before it diverges from the invariant
line. Figure~\ref{fig:gains2species} demonstrates the effect of
including the feedback term \eqref{simple-feedback} if suppression of
the unwanted species $2$ is imperfect ($b_{\min}=10^{-3}$). The
imperfect suppression is mimicked in our simulations by increasing
$b_i$ ($i=1,2$) to $10^{-3}$ after each integration step if its value
fell below $10^{-3}$ in this step. \js{The curves show the error
  relative to the difference between $\mu_1$ and $\mu_2$. For $G_1=0$
  (no feedback control) $\bar D$ in \eqref{simple-feedback} was varied
  to obtain different points approximating $\mu_1$, otherwise $\bar s$
  in \eqref{simple-feedback} was varied ($\bar s$ has no effect if
  $G_1=0$).}

The two approaches in Fig.~\ref{simu2species} and
Fig.~\ref{discrete2species} correspond to two different choices for
the trade-off between speed and accuracy. While the dynamic feedback
in Fig.~\ref{simu2species} requires only a single run, the procedure
of switching back and forth between regions as rapidly as possible is
able to obtain the growth rate of the suppressed species for abscissae
$s$ more distant from $s_c$.

\section{Conclusion, discussion and outlook}

In this work, we have presented a framework for the functional
identification of \ar{a large class} of non-monotonic growth functions in the chemostat. The
proposed methods achieve identification by tracing out branches of
equilibria also through their unstable parts. At the core of the
method is the observation that the introduction of a stabilizing
feedback loop transforms the problem of finding equilibria of the
original uncontrolled (open-loop) system to a root-finding problem,
which can then be solved using either continuous and step-wise
variants of classical numerical continuation algorithms
\cite{AG03}. Numerical simulations illustrate the potential of the
method \ar{on the Haldane function.} 

\js{An important issue in practice is how long it would typically take
  to identify the entire growth function in a real experiment. In our
  simple model \eqref{chemostat} with idealized feedback
  \eqref{simple-feedback} the control gain $G_1$ can be chosen
  arbitrarily large such that the identification could be sped up
  arbitrarily. In practice several effects place a limit on our choice
  of gain $G_1$, such as output and state disturbances, and low
  sampling frequency for measurement and input. Furthermore, both
  approaches (sections \ref{section_continuation} and
  \ref{sec:hybrid}) have a parameter controlling the trade-off between
  speed of the process and the accuracy of the results: the tracing
  speed $\epsilon$ of $\bar s$ in section~\ref{section_continuation}
  and the step-size $\delta$ in section~\ref{sec:hybrid} (larger steps
  result in a coarse mesh on which the growth function is determined).
  Further investigations are required to find out how these parameters
  have to be chosen in real experiments.}

The approach is more general than the case we have presented here for
the chemostat model.  We use the chemostat as a conceptually simple
example that is still of practical interest.

Another application we plan to explore in the future are regulation
problems. For example, one can regulate the single-species chemostat
to operate at the substrate concentration $s$ at which the growth rate
$\mu$ is maximal by following the same recipe. 
This approach to regulation, which is similar in spirit to the act-and-wait
technique for delay compensation \cite{I06}, does not require an
a priori identification of the growth rate $\mu$, and leads to a
different algorithm than the methods discussed in the literature
\cite{DPG12,GDM04}).




\end{document}